\documentclass[12pt]{article}
\usepackage{amsmath,amsthm,amssymb}
\usepackage{remark}
\usepackage{boxedminipage}

\begin{document}

\newremark{remark}{Remark}

\parindent0mm
\newtheorem{theorem}{Theorem}
\newtheorem{lemma}{Lemma}
\newtheorem{ex}{Example}
\newtheorem{cor}{Corollary}
\newtheorem{defn}{Definition}
\newcommand{\ignore}[1]{}

\title{\bf On the {\em {SPR}}ification of linear descriptor systems via output feedback}

\author{Martin Corless$^{1}$\thanks{$^{1}$School of Aeronautics \& Astronautics, Purdue University, West Lafayette, IN, USA}
Ezra Zeheb$^{2}$\thanks{$^{2}$Technion - Israel Institute of Technology and Holon Institute of Technology, Israel.}
Robert Shorten$^{3}$\thanks{$^{3}$ Department of Electrical and Electronic Engineering, University College Dublin, Belfield, Dublin 4, Ireland.}}%

\maketitle

\begin{abstract}

We consider input-output systems in descriptor form and ask when such systems can be rendered 
SPR (strictly positive real)
 via output feedback. Time and frequency domain conditions are given to determine
when and how this is possible. In addition, a synthesis procedure for controller design is also
derived. Together, the results provide a
complete answer to when a linear descriptor system can be made SPR via output feedback, and 
give a recipe for design of the feedback  controller when it exists. Simple examples are given to illustrate 
our results and to demonstrate their efficacy. 

\end{abstract}

\section{History and related work}
A passive linear time-invariant system is a system whose transfer function is
positive real. The notion of a positive real (PR) function and of a strict
positive real (SPR) function is important  in many areas of engineering systems, in
particular in control theory and in circuit theory. A partial list of such
areas would be design of passive filters, absolute stability theory, output
feedback stabilization of uncertain systems, adaptive control systems,
switching systems, and variable structure systems. An important historical property 
that has made passivity and positive realness particularly attractive to classical control engineers is that a physical system, which is passive,
has properties that makes the behavior of the system ``friendly''. For example, it is well known
that a negative feedback connection of two (strictly) passive systems is always
asymptotically stable, and that the stability of more complicated interconnections of 
passive systems is characterised by simple algebraic conditions \cite{berman, hill}.
This has made passivity very useful in the design of distributed control systems. 
More recently, passivity has assumed an important role in the study of optimisation 
algorithms and consensus \cite{arcak}, the study of cyber physical systems \cite{panos}, 
and in the exploration of diagonal stability problems \cite{berman}.\newline 

Despite the contemporary interest in passivity, it is worth noting that the concept of passivity, and its connection to positive realness, is an old one. 
The concept of positive realness was introduced by Brune \cite{Brune1931} about 85 years ago. Brune proved that the driving point impedance of every
passive electrical network is a positive real function and that every positive real function can be synthesized by a passive electrical network.
Bott and Duffin  \cite{BottDuffin1949} demonstrated that Brune's result also holds for  a passive network without transformers, that is,
networks consisting of inductors, capacitors and resistors only. An important modification of the original concept of positive realness is 
the notion of {\em strict positive realness}. Strict positive realness (SPR) was introduced in the control community via the KYP Lemma \cite{kyp1,kyp2,kyp3,kyp4}. 
The KYP (Kalman, Yacubovich, Popov) Lemma is a fundamental result  in system and control theory as it establishes a connection between frequency domain criteria and 
state space criteria.
This relationship has been central to the development of several areas of control - in particular in the study of  absolute stability theory \cite{kyp4, shorten},
adaptive control, switching systems and more.\newline

Given this general background, an important problem in control is to establish when it is possible to convert 
a system that is not SPR into one that is SPR via output feedback. This question has attracted the attention 
of many researchers over the past thirty years and their progress has been documented in a series 
of papers. 
While these are too numerous to mention explicitly, we mention a few here.
All of these results are for systems whose transfer function is  strictly proper.
Necessary and sufficient conditions 
for a  transfer function to be made SPR via  static output feedback are given in \cite{fradkov}.
For strictly proper transfer functions, it is shown in \cite{Huang} that,
if no static output feedback controller exists such that the transfer function of the closed loop system is SPR, 
then there does not exist an output dynamic
feedback  controller such that the transfer function of the closed loop system is SPR, as
well;
\cite{Huang} also contains  necessary and sufficient conditions  for the existence of a
static output feedback controller rendering a closed loop SPR transfer function.
These conditions depend on the existence of a positive definite matrix
complying with a certain matrix inequality.
In \cite{barkana} and \cite{owens},
necessary and sufficient conditions for the existence of a static output
feedback controller rendering a closed loop SPR transfer function are given. These
are expressed in terms of the transfer function of the open loop system.
\newline

Our objective in this paper is to ask when a general linear input-output  descriptor system, whose transfer function is not necessarily proper,
can be made strictly positive real   via output feedback. Classically, such questions 
often give rise to several types of equivalent characterisations: (i) in the time domain; 
(ii) in the frequency domain; as well as necessitating the need for a controller synthesis 
procedure.  Our conditions also
give rise to three such (equivalent) characterisations. The first is a time domain spectral (eigenvalue) condition; the 
second an equivalent frequency domain condition; and the third is essentially a control design procedure. 
Together, these conditions provide a
complete answer to when a general  linear  descriptor system can be made SPR via output feedback, and 
give a recipe for design of the controller when it exists. Simple examples are given to illustrate 
our results and to demonstrate their efficacy. \newline

{\bf Specific contributions:} The topic of when a linear system can be made  SPR
via output feedback has a rich history.  Given the volume of work on this topic, a brief comment 
on the contributions of this note is merited.  Specifically, our work is novel in a number of 
ways. First, we give conditions for  general linear descriptor systems   rather than the usual state space systems;  descriptor systems can have improper transfer functions.
Second,
to the best of our knowledge, our spectral and synthesis procedure have not been derived 
elsewhere in the literature. Third, for the first time, a complete set of equivalent system theoretic 
characterisations of SPRification is given (time domain, frequency domain) in one place.

\section{Problem statement and main results}

Consider an input-output system  in descriptor form described by
\begin{equation}
\label{eq:ss1}
\begin{array}{rcl}
E\dot{x} &=& Ax +B(u+ w)\\
y&=& Cx+D(u+w)
\end{array}
\end{equation}
where the {\sf state} $x(t)$ is an $n$-vector\footnote{An $n$-vector is a real or complex vector with $n$ components, that is, an element of
$\mathbb{R}^n$ or $\mathbb{C}^n$, respectively}, the {\sf  control input} $u(t)$, {\sf output} $y(t)$, and {\sf exogeneous input} $w(t)$ are $m$-vectors
while  $E$ and $A$ are $n\times n$ matrices and  $B, C,$ and $D$ are matrices of dimensions $n\times m, m\times n$ and
$m \times m$, respectively.
Sometimes we refer to this system as the system $(E, A, B, C, D)$.
We wish to know whether or not this system can be made stable and SPR (strictly positive real) through static output feedback, specifically, if there exist matrices $K$  and $L$ such that
the  system resulting  from
\begin{equation}
\label{eq:fbk}
u=Ky \qquad \mbox{and} \qquad z =Ly
\end{equation}
is stable and the resulting  transfer function from $w$ to $z$ is SPR.\newline

For the feedback law  in \eqref{eq:fbk} to be well-posed, one must require that $I-KD$  be non-singular.
In Calculation 1 in the Appendix, it is shown 
that, provided  $I-KD$ is non-singular,  the system resulting from control  (\ref{eq:fbk}) applied to system (\ref{eq:ss1}) is described by
\begin{equation}
\label{eq:descSS1}
\begin{array}{rcl}
E\dot{x} &=& A_cx +B_cw\\
z&=& C_cx +D_cw
\end{array}
\end{equation}
where
\begin{equation}
A_c =A +B(I-KD)^{-1}KC\,.
\end{equation}
%
 A number $\lambda$ is an {\sf eigenvalue} of the matrix pair $(E, A_c)$ or system (\ref{eq:descSS1})  if  there is a non-zero vector $v$ such that
$( \lambda E - A_c)v=0$.
Any such vector $v$ is  called an {\sf eigenvector} corresponding to $\lambda$.
The descriptor system (\ref{eq:descSS1}) is {\sf stable} or the pair $(E,A_c)$ is stable if  every  eigenvalue of $(E,A_c)$ has negative real part.
Throughout, we  assume that $(E, A)$ is regular in the sense that $\det(sE-A)$ is not identically zero.\newline


The transfer function associated with  the original system (\ref{eq:ss1}) is given by
\begin{equation}
G(s) = C(sE-A)^{-1}B + D \,.
\end{equation}
In the Laplace domain,  the original system and  the output feedback (\ref{eq:fbk}) are described by
\begin{equation}
\label{eq:fbkL}
\hat{y}=G(\hat{u}+\hat{w})\,, \qquad \hat{u}=K\hat{y} \,, \qquad \hat{z}=L\hat{y}
\end{equation}
where $\hat{u}, \hat{y}, \hat{w}, \hat{z}$ denote the Laplace transforms of $u, y, w, z$.
Thus
$\hat{y}=GK\hat{y} +G\hat{w}$, which implies that
 $(I-GK)\hat{y}= G\hat{w}$.
Hence,
$
\hat{z}=  G_c \hat{w}
$
where
\begin{equation}
\label{eq:G_c}
G_c =L(I-GK)^{-1}G= LG(I-KG)^{-1} 
\end{equation}

\vspace{1em}
We say that {\em $\lambda$ is a {\sf pole}  of  $G_c$  or  system (\ref{eq:descSS1}) if
$\lim_{s \rightarrow \lambda} G_c(s)$ does not exist.}\newline

{\em
{\bf Strict Positive Realness:} The transfer function $G_c$ is {\sf strictly positive real (SPR )} if there exists $\epsilon >0$ such that\footnote{ $M^\prime$ denotes the complex conjugate transpose of a  matrix $M$.} 
whenever $s\in \mathbb{C}$ is not a pole of $G_c$:
\begin{equation}
\label{eq:SPRdefn}
G_c(s) + G_c(s)^\prime >0 \quad \mbox{for} \quad \Re(s) \ge - \epsilon\,.
\end{equation}}

Now we can  state the problem under consideration in a more precise fashion.

\vspace{1em}
\begin{center}
\begin{boxedminipage}[t]{1.0\linewidth}
{\bf Problem statement:}
{\em Determine conditions under which there exist matrices $K$ and $L$ with $I-KD$ non-singular such that
\begin{eqnarray}
\label{eq:Objective1}
(E,\ A\!+\!B(I-KD)^{-1}KC)  & & \mbox{is stable}\\
\label{eq:Objective2}
L(I-GK)^{-1}G & & \mbox{is SPR}
\end{eqnarray}
}
\end{boxedminipage}
\end{center}


\subsection{Main Result (Part A): A spectral (time domain) characterisation of SPRification}
\label{sec:mainresA}
We can now present our first main result. To proceed, we shall need the following two matrices:
\begin{equation}
\label{eq:DescNew}
\mathcal{E} :=\left[\begin{array}{cc}
E  &0  \\
0   &0
\end{array}
\right]
\qquad \mbox{and} \qquad
\mathcal{A} :=\left[\begin{array}{cc}
A &B  \\
C   &D
\end{array}
\right]
\end{equation}
\vspace{1em}
and the following rank conditions:
{\em
\begin{equation}
\label{eq:fullRankInf}
\left[\begin{array}{cc} E & B \end{array} \right]
\qquad\mbox{and} \qquad
\left[\begin{array}{c} E \\C \end{array} \right] \\ 
\end{equation}
have maximum rank.
}

Our first main result involves   a simple  eigenvalue  condition on the state space matrices describing the system.\footnote{Recall that the index of zero as an eigenvalue of   a  square matrix  $S$   is the smallest   integer $k\ge 0$  for which
the rank of $S^{k+1}$ equals the rank of ${S}^{k}$.}
\begin{center}
\begin{boxedminipage}[t]{1.0\linewidth}
\begin{theorem}
\label{th:main}
Consider a system described by (\ref{eq:ss1})
that satisfies  rank condition (\ref{eq:fullRankInf}). %
There exist matrices $K$ and $L$, with $I-KD$ non-singular, such that objectives (\ref{eq:Objective1}) and (\ref{eq:Objective2}) hold if and only if  the following conditions hold.
\begin{itemize}
\item[(a)]
$ \mathcal{A}$ is non-singular and the non-zero eigenvalues of $\mathcal{A}^{-1}\mathcal{E}$ have negative real part.
\item[(b)] 
The index of zero as an eigenvalue\
of $\mathcal{A}^{-1}\mathcal{E}$ is at most two.
\end{itemize}
\end{theorem}
\end{boxedminipage}
\end{center}

\vspace{1em}
\begin{ex}{[Simple integrator]}
{\rm
Consider the simple integrator described by $\dot{x} = u$ and $y=x$.
Here $E=1$, $A=0$, $B=C=1$ and $D=0$.
Hence,
\[
\mathcal{A} =\left[\begin{array}{cc}
0	&	1\\
1	&  0
\end{array}
\right],
\mathcal{E} =\left[\begin{array}{cc}
1	&	0\\
0	&  0
\end{array}
\right],
\mathcal{A} ^{-1}\mathcal{E} =
\left[\begin{array}{cc}
0	&	0  \\
1	&  0
\end{array}
\right]
\]
So, $\mathcal{A} ^{-1}\mathcal{E}$ only has an eigenvalue at zero and its index is two.
Hence by Theorem \ref{th:main},
objectives (\ref{eq:Objective1}) and (\ref{eq:Objective2}) 
can be achieved.
}
\end{ex}


\begin{ex}{[Descriptor system]}
{\rm
Consider the descriptor system with
{\normalsize
\[
E=\left[\begin{array}{rr}
0 &0  \\
1   &0
\end{array}
\right],
A=\left[\begin{array}{rr}
1 &0  \\
0   &1
\end{array}
\right]  \] 
\[
B=\left[\begin{array}{rr}
-1  &-1  \\
0   &0
\end{array}
\right],	
C=\left[\begin{array}{rr}
-1  &1 \\
0   &1
\end{array}
\right],
D=\left[\begin{array}{rr}
-3  &0  \\
-1   &0
\end{array}
\right]
\]
}
Since $\det(sE-A)\equiv 1$, this system has no finite eigenvalues.
Here
\[
\mathcal{A}=
\left[\begin{array}{rrrr}
1&0          &-1   &-1  \\
0       &1   &0   &0   \\
-1      &1     &-3   &0 \\
0       &1      &-1  &0
\end{array}
\right] 
\] \[
\mathcal{E}=
\left[\begin{array}{rrrr}
0  &0      &0  &0  \\
1       &0   &0   &0   \\
0    &0      &0  &0 \\
0       &0      &0   &0
\end{array}
\right]
,
\quad\mathcal{A}^{-1}\mathcal{E}=
\left[\begin{array}{rrrr}
    -2  &   0   &  0   &  0 \\
     1 &    0    & 0   &  0 \\
     1  &   0   &  0   &  0 \\
    -3  &   0    & 0    & 0
    \end{array}
\right]
\]
So, $\mathcal{A}^{-1}\mathcal{E}$ has  eigenvalues $0$ and $-2$ and the order of the zero eigenvalue is
one, it follows from Theorem \ref{th:main} that
objectives (\ref{eq:Objective1}) and (\ref{eq:Objective2})  can be achieved for this system.
}
\end{ex}

\subsection{Main Result (Part B): A frequency domain characterisation of SPRification}
\label{sec:mainresB}
Our next result, which involves conditions on the transfer function of the system, needs the following definitions
 for a finite complex number $\lambda$.



\vspace{1em}
{\em $\lambda$ is an {\sf uncontrollable eigenvalue}  for  $(E, A, B)$  or system (\ref{eq:ss1}) if
\[
\left[\begin{array}{cc} \lambda E\!-\!A & B\end{array} \right]
\]
does not have maximum rank.
}

\vspace{1em}
{\em
$\lambda$ is an {\sf unobservable eigenvalue} for $(E, C, A)$ or  system (\ref{eq:ss1})  if
\[
\left[\begin{array}{c} \lambda E\!-\!A \\ C\end{array} \right]
\]
does not have maximum rank.
}

\vspace{1em}
{\em If infinity is a pole of   a transfer function $H$, its  {\sf order}  is  the smallest integer  $l$ for which  $\lim_{s\rightarrow  \infty} s^{-l}H(s)$ exists.}

\begin{center}
\begin{boxedminipage}[t]{1.0\linewidth}
\begin{theorem}
\label{th:main2}
Consider a system described by (\ref{eq:ss1}).
%
There exist matrices $K$ and $L$, with $I-KD$ non-singular, such that objectives (\ref{eq:Objective1}) and (\ref{eq:Objective2}) hold if and only if  the following conditions hold.
\begin{itemize}
\item[(a)] The finite poles of $G^{-1}$ have negative real part where $G$ is the system transfer function.
\item[(b)] If $G^{-1}$ has a pole at infinity then, its order  is one.
\item[(c)]
The uncontrollable and unobservable  eigenvalues of the system   have negative real part.
\end{itemize}
\end{theorem}
\end{boxedminipage}
\end{center}

\vspace{.5em}
\begin{ex}{\rm
Consider the system with $E=I$ and 
\[
A=\left[\begin{array}{rr}
-1  &0  \\
0   &1
\end{array}
\right]\,, \qquad
B=\left[\begin{array}{rr}
1  &0  \\
0   &1
\end{array}
\right]\,,\]\[
C=\left[\begin{array}{rr}
-2  &0  \\
0   &2
\end{array}
\right]\,, \qquad
D=\left[\begin{array}{rr}
1  &0  \\
0   &1
\end{array}
\right]
\]
for which
\[
G(s)=\left[\begin{array}{cc}
\displaystyle{\frac{s-1}{s+1}}  &0  \\
0   &\displaystyle{\frac{s+1}{s-1}}
\end{array}
\right]\,\]
and
\[
G(s)^{-1}=\left[\begin{array}{cc}
\displaystyle{\frac{s+1}{s-1}}  &0  \\
0   &\displaystyle{\frac{s-1}{s+1}}
\end{array}
\right].
\]
So $G^{-1}$ has poles at $-1$ and $1$.
Since $G^{-1}$ has a pole at one,  it follows from Theorem \ref{th:main2} that
objectives (\ref{eq:Objective1}) and (\ref{eq:Objective2}) 
cannot be achieved for this system.
Note that $
\det[G(s)] \equiv  1
$.
Here
\ignore{
\[
\det(s \mathcal{E} - \mathcal{A}) =
\left[\begin{array}{cccc}
-1 -s   &0      &1   &0  \\
0       &1-s    &0   &1   \\
-2       &0      &1   &0 \\
0       &2      &0   &1 \\
\end{array}
\right]
\]
}
\[
\mathcal{E} = 
\left[\begin{array}{cccc}
1  &0      &0   &0  \\
0       &1  &0   &0   \\
0      &0      &0  &0 \\
0       &0     &0   &0\\
\end{array}
\right]
\,, \]\[
\mathcal{A} = 
\left[\begin{array}{rrrr}
-1   &0      &1   &0  \\
0       &1    &0   &1   \\
-2       &0      &1   &0 \\
0       &2      &0   &1 \\
\end{array}
\right],
\quad
\mathcal{A}^{-1}\mathcal{E} = 
\left[\begin{array}{rrrr}
     1  &   0  &   0    & 0		\\
     0   & -1   &  0   &  0		\\
     2    & 0   &  0   &  0		\\
     0 &    2    & 0   &  0
\end{array}
\right]
\]
So, $\mathcal{A}^{-1}\mathcal{E}$ has eigenvalues $-1, 1$ and $0$.
 }
\end{ex}
\begin{ex}{[$A$ and $D$ singular]}
{\rm
Consider a system with $E=I$ and 
\[
A=\left[\begin{array}{rr}
0 &0  \\
0   &-1
\end{array}
\right]\,, \qquad
B=\left[\begin{array}{rr}
1  &0  \\
0   &1
\end{array}
\right]\,\]\[
C=\left[\begin{array}{rr}
1  &0  \\
0   &1
\end{array}
\right]\,, \qquad
D=\left[\begin{array}{rr}
1  &0  \\
0   &0
\end{array}
\right]
\]
which  has no uncontrollable or unobservable eigenvalues,
and for which
\[
G(s)=\left[\begin{array}{cc}
\displaystyle{\frac{s+1}{s}}  &0  \\
0   &\displaystyle{\frac{1}{s+1}}
\end{array}
\right],\]
and \[
G(s)^{-1}=\left[\begin{array}{cc}
\displaystyle{\frac{s}{s+1}}  &0  \\
~\\
0   &\displaystyle{{s+1}}
\end{array}
\right]\,.
\]
So, $G^{-1}$ has a single finite pole at $-1$  which is also an eigenvalue of $A$.
Since the order of infinity as a pole of $G^{-1}$ is one, it follows from Theorem \ref{th:main2} that
objectives (\ref{eq:Objective1}) and (\ref{eq:Objective2})  can be achieved for this system. Here
\[
\mathcal{A}=
\left[\begin{array}{rrrr}
0&0      &1   &0  \\
0       &-1   &0   &1   \\
1      &0      &1   &0 \\
0       &1      &0   &0
\end{array}
\right]
\,, \qquad
\mathcal{E}=
\left[\begin{array}{rrrr}
 1   &0      &0  &0  \\
0       &1   &0   &0   \\
0    &0      &0  &0 \\
0       &0      &0   &0
\end{array}
\right]\]\[
\mathcal{A}^{-1}\mathcal{E}=
\left[\begin{array}{rrrr}
 -1   &0      &0  &0  \\
0       &0   &0   &0   \\
1   &0      &0  &0 \\
0       &1      &0   &0
\end{array}
\right]
\]
So, $\mathcal{A}^{-1}\mathcal{E}$ has eigenvalues $-1$ and $0$ and the order of zero is two.
}
\end{ex}

\subsection{Main result (Part C): Controller construction for SPRification}
\label{sec:mainresC}
The next result provides a simple method for controller construction.
\begin{center}
\begin{boxedminipage}[t]{1.0\linewidth}
\begin{theorem}
\label{th:main3}
Consider a system described by (\ref{eq:ss1}).
There exist matrices $K$ and $L$, with $I-KD$ non-singular, such that objectives (\ref{eq:Objective1}) and (\ref{eq:Objective2}) hold if and only if  the following conditions hold.
\begin{itemize}
\item[(a)]
\begin{equation}
\label{eq:GinvRat}
G(s)^{-1} = s H_1 + D_2 + R(s)
\end{equation}
where $H_1$  and $D_2$ are  constant matrices and 
either
\begin{itemize}
\item[(i)] $R=0$
or
\item[(ii)]
 \begin{equation}
 \label{eq:R}
R(s) =  C_2(sI-A_2)^{-1}B_2
\end{equation}
  where $A_2, B_2, C_2$ are constant matrices with  $A_2$  Hurwitz.
\end{itemize}
\item[(b)]
The uncontrollable and unobservable  eigenvalues of the system   have negative real part.
\end{itemize}

In either case, the gain matrices $L$ and  $K$, that achieve the desired objectives, are obtained as follows.
$L$ is any invertible matrix  that is chosen so that $LH_1$ is symmetric positive semi-definite and 
\begin{equation}
\label{eq:K}
K= D_2 -L^{-\prime}N
\end{equation}
where   $N$  is any matrix that satisfies:

\begin{itemize}
\item[(i)] $R=0$:
\[
N+N' > 0
\]
\item[(ii)] $R(s) = C_2(sI-A_2)^{-1}B_2$:
 \begin{equation}
 \label{eq:Ncond2}
 N+N' > (PB_2 - C_2' L)'Q^{-1}(PB_2 - C_2' L)
 \end{equation}
 where   $Q$ is any symmetric positive-definite matrix and 
  $P$ is the unique solution to
 \begin{equation}
 \label{eq:lyap}
 PA_2+ A_2'P + Q =0
 \end{equation}

 \end{itemize}
 \end{theorem}
 \end{boxedminipage}
 \end{center}
 
 \vspace{1em}
 {\bf Comment:} A method for computing $A_2, B_2, C_2, D_2$ and $H_1$ is given in Section \ref{subsec:controller}.
 That section also contains a method for computing $L$.

\begin{ex}
\label{ex5}
{\rm
Consider the system with 
transfer function
\[
G(s) = \frac{s+1}{s-2} 
\]
This is a system with a controllable and observable state space realization given by
$E=1, A=2, B=1, C=3, D=1$.
Here
\[
G(s)^{-1} = \frac{s-2} {s+1} = 1 +\frac{-3} {s+1}
\]
Hence,
$G^{-1}$ can be expressed as in  \eqref{eq:GinvRat}   and \eqref{eq:R} with $A_2 =-1, B_2=1, C_2 = -3, D_2 =1$ and $H_1 = 0$.
Since $A_2$ is Hurwitz, it  follows from Theorem \ref{th:main3} that
objectives (\ref{eq:Objective1}) and (\ref{eq:Objective2}) 
can be achieved for this system.
Since $H_1 = 0$, $L$ can be any nonzero number.
Using \eqref{eq:K}, \eqref{eq:Ncond2} and \eqref{eq:lyap}  yields  $K=1 -N/L$,
$P=Q/2$ and 
\[
N>(Q/2 +3L)^2/2Q
\]
Considering $L>0$ we obtain
$
K< 1 -(Q+6L)^2/8LQ  
$.
The maximum value of the righthand side of the above inequality is $-2$  which occurs with $L^{-1}Q=6$.
Thus, $L=1$ and any $K<-2$ achieves the desired results. Considering $L<0$ we obtain
$
K> 1 -(Q+6L)^2/8LQ  
$.
The minimum value of the righthand side of the above inequality is $1$ occurs with $L^{-1}Q=-6$.
Thus, $L=-1$ and any $K> 1$ also achieves the desired results.

}
\end{ex}


\subsection{Dynamic output feedback}
Here we show that 
{\em if one can achieve an internally stable SPR system  using  a proper dynamic controller
then, this can be achieved with a static controller.}

To see this,  consider  a general dynamic output feedback controller described by
\begin{equation}
\hat{u} = K \hat{y} \qquad\mbox{and} \qquad  \hat{z} = L \hat{y}
\end{equation}
where $K$ and $L$ are   rational  transfer functions.
We assume that $K$ and $L$ are proper in the sense that they have no poles at infinity.
Since $\hat{y} =G(\hat{u}+\hat{w})$, the transfer function from $\hat{w}$ to $\hat{z}$ is given by
\begin{equation}
\label{eq:Gc2}
 \hat{z} = G_c\hat{w} \quad \mbox{where} \quad  G_c = L(I-GK)^{-1}G
\end{equation}
Suppose $G_c$ is SPR and the closed-loop system is internally stable.
Then we claim that the finite poles of  $K$ and $L$ have negative real parts.

To see this, consider the closed-loop system with additional inputs $w_2$ and $w_3$:
\begin{align*}
\hat{y}= G(\hat{u}+\hat{w}), \quad
\hat{u}=K(\hat{y}+\hat{w}_2),\quad 
\hat{z} = L(\hat{y}+\hat{w}_3) 
\end{align*}
Then
\begin{equation}
\hat{z}=G_c \hat{w}+ G_c K\hat{w}_2 +L\hat{w}_3 \
\end{equation}
Internal stability of the closed loop system requires that  all the finite poles of  $G_cK$  and $L$  have negative real part.
With $G_c$ being SPR, $G_c^{-1}$ is also SPR and all its finite poles have negative real part.
Since $K=G_c^{-1}(G_cK)$ it now follows that all its finite poles of $K$ have negative real part.

From \eqref{eq:Gc2} we see that 
$G_c^{-1} =(G^{-1} -K)L^{-1}
$; hence
\begin{equation}
G^{-1} = G_c ^{-1}L+ K
\end{equation}
Since all the finite  poles of $L$, $K$ and $G_c^{-1}$ have negative real part, it now follows that all the finite poles of $G^{-1}$
have negative real part.
Also, since $G_c$ is SPR, it can have a most one pole at infinity.
Since $L$ and $K$ have no poles at infinity, it now follows that $G^{-1}$ has at most one pole at infinity.

Since output feedback does not affect uncontrollable and unobservable eigenvalues,
these eigenvalues must have negative real part for internal stability.
It now follows from Theorem \ref{th:main2} that the system can be made SPR  and internally stable with static output feedback.

Note that \cite{Huang} obtained the same result for systems with strictly proper transfer functions, that is systems with $E=I$ and $D=0$, using different proof techniques.

\section{Implications of main results}
 \subsection{Zero output dynamics}
Conditions (a) and (b) of Theorem \ref{th:main} have a nice interpretation in terms of the zero output dynamics of system \eqref{eq:ss1}.
The {\sf  zero output dynamics} of system \eqref{eq:ss1} are those dynamics which result when $w=0$ and the control input is chosen to keep the output precisely zero.
It follows from \eqref{eq:ss1} that these dynamics are described by
\begin{equation}
\label{eq:zeroOut}
\begin{array}{rcl}
E\dot{x} &=& Ax +Bu\\
0&=& Cx+Du
\end{array}
\end{equation}
that is, they are described by the descriptor system
 \begin{equation}
\mathcal{E} 
\left[\begin{array}{c}
\dot{x}\\\dot{u}
\end{array} \right]
=
\mathcal{A} 
\left[\begin{array}{c}
x \\ u
\end{array} \right]
\,.
\end{equation}
Thus,  the zero output dynamics are determined by the descriptor system characterized by  $(\mathcal{E}, \mathcal{A})$.
The requirement that $\mathcal{A}$ be non-singular is equivalent to $(\mathcal{E}, \mathcal{A})$ not having a zero eigenvalue.
In this case,  the  eigenvalues of  $(\mathcal{E}, \mathcal{A})$ are the inverse of the non-zero eigenvalues of $\mathcal{A}^{-1}\mathcal{E}$.
Hence, condition (a) of Theorem \ref{th:main} is equivalent to  the eigenvalues of $(\mathcal{E}, \mathcal{A})$  having negative real part,
that is, the zero output dynamics system is  asymptotically  stable.\newline

To obtain an interpretation of condition (b) of Theorem \ref{th:main}
we have the folllowing definition when the matrix $\mathcal{A}$ is invertible \cite{SaiijaAl2013Descriptor}.

\vspace{1em}
{\em The {\sf index }of  $( \mathcal{E}, \mathcal{A})$ or system (\ref{eq:zeroOut}) is the  smallest integer $k\ge0$ for which
the rank of 
$(\mathcal{A}^{-1}\mathcal{E})^{k+1}$
equals the rank of 
$(\mathcal{A}^{-1}\mathcal{E})^{k}$.}

\vspace{1em}
Since $\mathcal{E}$ is singular, $\mathcal{A}^{-1}\mathcal{E}$ has a zero eigenvalue; hence the index of $( \mathcal{E}, \mathcal{A})$ is at least one and
it equals the index of zero as an eigenvalue of $\mathcal{A}^{-1}\mathcal{E}$.
The above discussion results in the following corollary of Theorem \ref{th:main}.

\begin{cor}
Consider a system described by (\ref{eq:ss1})
that satisfies  rank condition (\ref{eq:fullRankInf}).
There exist matrices $K$ and $L$, with $I-KD$ non-singular, such that objectives (\ref{eq:Objective1}) and (\ref{eq:Objective2}) hold if and only if 
  the zero output dynamics (\ref{eq:zeroOut}) are stable and have a maximum index of two.
\end{cor}

%

\subsection{Eigenvalues of $\mathcal{A}^{-1}\mathcal{E}$}
If $E=0$ then $\mathcal{E}$ is zero; hence, $\mathcal{A}^{-1}\mathcal{E}$ is zero which
means that zero is the only eigenvalue  of $\mathcal{A}^{-1}\mathcal{E}$  and its index is one.\newline

If $E\neq0$, let
$(X, Y)$ be any full rank decomposition of $E$, that is,
   \begin{equation}
   E=XY^{\prime}\end{equation}
    where $X$ and $Y$ are full rank matrices \cite{SaiijaAl2013Descriptor}.
Also, let
\begin{equation}
\label{eq:invcalA0}
\left[\begin{array}{cc}
\tilde{A}	&\tilde{B}	\\
\tilde{C}	&\tilde{D}
\end{array}
\right]
=
\mathcal{A}^{-1}
\end{equation}
where $\tilde{A}$ has the same dimensions as $A$ and let
    \begin{equation}
\label{eq:Atilde}
E_1= Y'\tilde{A}X
\,.
\end{equation}
The next result which is proven in the Appendix provides relationships between the eigenvalues of $\mathcal{A}^{-1}\mathcal{E}$  and $E_1$ and is useful in checking conditions (a) and (b) of Theorem \ref{th:main}. 
\begin{lemma}
\label{lem:reduce}
When $\mathcal{A}$ is non-singular and $E\neq 0$, the non-zero eigenvalues of  $\mathcal{A}^{-1}\mathcal{E}$  and $E_1$ are the same 
and  the index of zero\footnote{If a matrix does not have an eigenvalue at zero we say the index of zero as an eigenvalue is zero.} as eigenvalue of $E_1$ equals $l-1$,
 where
  the index of zero as an eigenvalue of $\mathcal{A}^{-1}\mathcal{E}$ is $l$.
  \end{lemma}

\ignore{
Then
\[
\mathcal{A}^{-1}\mathcal{E} =
\left[\begin{array}{cc}
\tilde{A}E	&0\\
* &0
\end{array}
\right]
\]
Hence the non-zero eigenvalues of $\mathcal{A}^{-1}\mathcal{E}$ equal the nonzero eigenvalues of $\tilde{A}E$.
Also,
\[
(\mathcal{A}^{-1}\mathcal{E})^k =
\left[\begin{array}{cc}
(\tilde{A}E)^k	&0\\
* &0
\end{array}
\right]
\]
Hence   condition (b) of Theorem \ref{th:main} is equivalent to the following condition:
If zero is an eigenvalue of $\tilde{A}E$ then  its index is at most two.
These considerations yield the following corollary to Theorem \ref{th:main}.
%


\begin{cor}
Suppose that the following conditions hold for  system (\ref{eq:ss1}).
\begin{itemize}
\item[(a)] $\mathcal{A}$ is non-singular.
\item[(b)] The  non-zero eigenvalues of $E_1$ have negative real part.
\item[(c)] If zero is an eigenvalue of  $E_1$  then its index is at most
one.
\end{itemize}
Then, there exist matrices $K$ and $L$, with $I-KD$ non-singular, such that objectives (\ref{eq:Objective1}) and (\ref{eq:Objective2}) hold.

Conversely, suppose that there exist matrices $K$ and $L$, with $I-KD$ non-singular, such that objectives  (\ref{eq:Objective1}) and (\ref{eq:Objective2}) hold  and 
rank condition (\ref{eq:fullRankInf}) holds.
Then  conditions (a),  (b) and (c) hold.
\end{cor}
 }


\subsection{Non-descriptor (non-algebraic) systems}

Recall that a descriptor system is defined by both differential equations and a set of algebraic equations. 
In our context, both of these  constraints 
are captured via the structure of the $E$ matrix. Classical systems whose dynamics 
are described by differential equations only give rise to an invertible $E$ matrix. Such systems 
are sometimes referred to as   normal systems. Since this  term is 
itself loaded and has different meanings in different areas of systems theory, we shall 
refer to a system with an invertible $E$ matrix as a {\sf non-algebraic} system.\newline 

In this case, without loss of generality, we can consider $E=I$ and such a system is described by
\begin{equation}
\label{eq:ssn}
\begin{array}{rcl}
\dot{x} &=& Ax +B(u +w)\\
y&=& Cx+D(u+w)
\end{array}
\end{equation}
Also, rank conditions (\ref{eq:fullRankInf}) hold. 
If we consider $(I,I)$ as a full rank decomposition of $E=I$ then, $E_1 = \tilde{A}$ and, using Lemma \ref{lem:reduce},
we have the following corollary to Theorem \ref{th:main}.

\vspace{1em}

\begin{cor}
\label{cor:main}
%
There exist matrices $K$ and $L$, with $I-KD$ non-singular,  such that objectives (\ref{eq:Objective1}) and (\ref{eq:Objective2}) hold  for a non-algebraic system described by  (\ref{eq:ssn})
if and only if  the following conditions hold.
\begin{itemize}
\item[(a)] $\mathcal{A}$ is non-singular.
\item[(b)] The non-zero eigenvalues of $\tilde{A}$ have negative real part.
\item[(c)] If zero is an eigenvalue of  $\tilde{A}$  then its index is 
one.
\end{itemize}
\end{cor}

Conditions (b) and (c) above are equivalent to saying that the system $\dot{z} = \tilde{A}z$ is marginally stable.

\subsubsection{Special cases}
\paragraph{ $D$ non-singular}
If $D$ is non-singular then, $ \mathcal{A}$  is non-singular if and only if $A-BD^{-1}C$ is non-singular. In this case
\begin{equation}
\tilde{A}=[A-BD^{-1}C]^{-1}
\,.
\end{equation}
Hence $\tilde{A}$ is invertible and does not have a zero eigenvalue.
If,  in addition,  the system is non-algebraic then, the requirements for objectives (\ref{eq:Objective1}) and (\ref{eq:Objective2})  to be achieved simplify to:

\vspace{1em}
{\em All the  eigenvalues of $A-BD^{-1}C$ have negative real part.}

\vspace{1em}
\paragraph{ $A$ non-singular}
If $A$ is non-singular then, $ \mathcal{A}$  is non-singular if and only if the matrix $D-CA^{-1}B$ is non-singular. In this case
\begin{equation}
\tilde{A}= A^{-1} +A^{-1}B(D-CA^{-1}B)^{-1}CA^{-1}
\end{equation}

\paragraph{ Strictly proper systems: $D=0$}
One can readily show that condition (c) is  equivalent to   $CB$ being invertible.

\ignore{
\subsection{$D=0$}
In this case,
\begin{equation}
\mathcal{N}(M) = \mathcal{R}(B)
\qquad \mbox{and} \qquad
\mathcal{R}(M) = \mathcal{N}(C)
\end{equation}
where $\mathcal{N}(M)$  denotes the null space or kernel of $M$ and $\mathcal{R}(M)$ denotes the range of $M$.

\vspace{1em}
{\bf To be completedÉÉ}

}


\ignore{

\section{Illustrative examples}

%

\begin{ex}{\rm
Consider a simple integrator with transfer function $G(s)= 1/s$.
$G^{-1}$ has a single pole which is infinity and the order of this pole is one.
Hence by Theorem  \ref{th:main2} objectives  can be met.
This system  has a state space realization with $E=1$, $A=0$, $B=C=1$ and $D=0$.
Hence
\[
\mathcal{A} =\left[\begin{array}{cc}
0	&	1\\
1	&  0
\end{array}
\right]
\,, \qquad 
\mathcal{E} =\left[\begin{array}{cc}
1	&	0\\
0	&  0
\end{array}
\right]
\qquad 
\mathcal{A} ^{-1}\mathcal{E} =
\left[\begin{array}{cc}
0	&	0  \\
1	&  0
\end{array}
\right]
\]
Here $\mathcal{A} ^{-1}\mathcal{E}$ only has an eigenvalue at zero and its index is two.
}
\end{ex}

\begin{ex}{\rm
Consider the normal system with
\[
A=\left[\begin{array}{rr}
-1  &0  \\
0   &1
\end{array}
\right]\,, \qquad
B=\left[\begin{array}{rr}
1  &0  \\
0   &1
\end{array}
\right]\,, \qquad
C=\left[\begin{array}{rr}
-2  &0  \\
0   &2
\end{array}
\right]\,, \qquad
D=\left[\begin{array}{rr}
1  &0  \\
0   &1
\end{array}
\right]
\]
for which
\[
G(s)=\left[\begin{array}{cc}
\displaystyle{\frac{s-1}{s+1}}  &0  \\
0   &\displaystyle{\frac{s+1}{s-1}}
\end{array}
\right]\,
\qquad
\mbox{and}
\qquad
G(s)^{-1}=\left[\begin{array}{cc}
\displaystyle{\frac{s+1}{s-1}}  &0  \\
0   &\displaystyle{\frac{s-1}{s+1}}
\end{array}
\right]\,.
\]
So $G^{-1}$ has poles at $-1$ and $1$.
Since $G^{-1}$ has a pole at one,  it follows from Theorem \ref{th:main2} that
objectives (\ref{eq:Objective1}) and (\ref{eq:Objective2}) 
cannot be achieved for this system.
Here
\ignore{
\[
\det(s \mathcal{E} - \mathcal{A}) =
\left[\begin{array}{cccc}
-1 -s   &0      &1   &0  \\
0       &1-s    &0   &1   \\
-2       &0      &1   &0 \\
0       &2      &0   &1 \\
\end{array}
\right]
\]
}
\[
\mathcal{E} = 
\left[\begin{array}{cccc}
1  &0      &0   &0  \\
0       &1  &0   &0   \\
0      &0      &0  &0 \\
0       &0     &0   &0\\
\end{array}
\right]
\,, \qquad 
\mathcal{A} = 
\left[\begin{array}{rrrr}
-1   &0      &1   &0  \\
0       &1    &0   &1   \\
-2       &0      &1   &0 \\
0       &2      &0   &1 \\
\end{array}
\right]
\]
and
\[
\det[s \mathcal{E} - \mathcal{A}] = s^2-1 =(s+1)(s-1).
\]
This illustrates Corollary  \ref{cor:cor2} which equates the finite poles of $G^{-1}$ to the eigenvalues of $(\mathcal{E}, \mathcal{A})$
when there are no uncontrollable or unobservable eigenvalues.

Also,  
\[
A-BD^{-1}C = \left[\begin{array}{rr}
1   &0\\
0   & -1
\end{array}
\right]
\]
Thus the eigenvalues of $A-BD^{-1}C$ are the same as the  poles of $G^{-1}$.
Moreover,
\[
E_1 = A^{-1} +A^{-1}B(D-CA^{-1}B)^{-1}CA^{-1}= \left[\begin{array}{rr}
1   &0\\
0   & -1
\end{array}
\right]
\]
Thus the inverse of the nonzero eigenvalues of the above matrix are the same as the  poles of $G^{-1}$
Note that
$
\det[G(s)] = 1
$.

}
\end{ex}

\begin{ex}{[$A$ and $D$ singular.]}
{\rm
Consider the normal system,
\[
A=\left[\begin{array}{rr}
0 &0  \\
0   &-1
\end{array}
\right]\,, \qquad
B=\left[\begin{array}{rr}
1  &0  \\
0   &1
\end{array}
\right]\,, \qquad
C=\left[\begin{array}{rr}
1  &0  \\
0   &1
\end{array}
\right]\,, \qquad
D=\left[\begin{array}{rr}
1  &0  \\
0   &0
\end{array}
\right]
\]
which  has no uncontrollable or unobservable eigenvalues
and for which
\[
G(s)=\left[\begin{array}{cc}
\displaystyle{\frac{s+1}{s}}  &0  \\
0   &\displaystyle{\frac{1}{s+1}}
\end{array}
\right]\,
\qquad
\mbox{and}
\qquad
G(s)^{-1}=\left[\begin{array}{cc}
\displaystyle{\frac{s}{s+1}}  &0  \\
~\\
0   &\displaystyle{{s+1}}
\end{array}
\right]\,.
\]
So, $G^{-1}$ has a single finite pole at $-1$  which is also an eigenvalue of $A$.
Since the order of infinity as a pole of $G^{-1}$ is one, it follows from Theorem \ref{th:main2} that
objectives (\ref{eq:Objective1}) and (\ref{eq:Objective2})  can be achieved for this system,

Here
\[
\mathcal{A}=
\left[\begin{array}{rrrr}
0&0      &1   &0  \\
0       &-1   &0   &1   \\
1      &0      &1   &0 \\
0       &1      &0   &0
\end{array}
\right]
\,, \qquad
\mathcal{E}=
\left[\begin{array}{rrrr}
 1   &0      &0  &0  \\
0       &1   &0   &0   \\
0    &0      &0  &0 \\
0       &0      &0   &0
\end{array}
\right]
\]
and
\[
\det[s\mathcal{E} - \mathcal{A}]  =s+1\,.
\]
Thus the finite  poles of $G(s^{-1})$ are the same as the eigenvalues of $(\mathcal{E}, \mathcal{A})$.

Also,
\[
E_1  = \left[\begin{array}{rr}
-1   &0\\
0   & 0
\end{array}
\right]
\]
Thus the inverse of the nonzero eigenvalues of the above matrix are the same as the  poles of $G(s^{-1})$.
Also, zero is a non-defective eigenvalue of $E_1$.

}
\end{ex}


\begin{ex}{[Descriptor system].}
{\rm
Consider the descriptor system,
{\small
\[
E=\left[\begin{array}{rr}
0 &0  \\
1   &0
\end{array}
\right],	\ \
A=\left[\begin{array}{rr}
1 &0  \\
0   &1
\end{array}
\right],  \ \
B=\left[\begin{array}{rr}
-1  &-1  \\
0   &0
\end{array}
\right],	 \ \
C=\left[\begin{array}{rr}
-1  &1 \\
0   &1
\end{array}
\right]\,, \ \
D=\left[\begin{array}{rr}
-3  &0  \\
-1   &0
\end{array}
\right]
\]
}
Since $\det(sE-A)= 1$, this system has no finite eigenvalues.

Here
\[
G(s)=\left[\begin{array}{cc}
s-4  &s-1 \\
s-1   &s
\end{array}
\right]\,
\qquad
\mbox{and}
\qquad
G(s)^{-1}=\frac{1}{2s+1}
\left[\begin{array}{cc}
-s  &	s-1\\
s-1  &	-s+4
\end{array}
\right]\,.
\]
The transfer function $G^{-1}$ has a single finite pole at $-1/2$ and no pole at infinity.
It now follows from Theorem \ref{th:main2} that
objectives (\ref{eq:Objective1}) and (\ref{eq:Objective2})  can be achieved for this system.

Here
\[
\mathcal{A}=
\left[\begin{array}{rrrr}
1&0          &-1   &-1  \\
0       &1   &0   &0   \\
-1      &1     &-3   &0 \\
0       &1      &-1  &0
\end{array}
\right]
\,, \qquad
\mathcal{E}=
\left[\begin{array}{rrrr}
0  &0      &0  &0  \\
1       &0   &0   &0   \\
0    &0      &0  &0 \\
0       &0      &0   &0
\end{array}
\right]
\]
and
\[
\det[s\mathcal{E} - \mathcal{A}]  =-(2s+1)\,.
\]
Thus the finite  pole of $G^{-1}$ is the same as the eigenvalue of $(\mathcal{E}, \mathcal{A})$.

Considering
\[
X=\left[\begin{array}{c}
0\\1
\end{array}\right]
\quad \mbox{and} \quad
Y=\left[\begin{array}{c}
1\\0
\end{array}\right]
\]
we obtain that
\[
E_1  =-2
\]
Thus the inverse of the  eigenvalue of $E_1$ is the same as the finite  pole of $G^{-1}$.
Since $G^{-1}$ is proper, $E_1$ has no zero eigenvalue.

}
\end{ex}
}

\section{SPRification in the frequency domain} 

In this section we obtain the transfer function characterization (Theorem \ref{th:main2}) of the systems for which the desired objectives can be achieved.
This involves  conditions on the poles of $G^{-1}$  and the uncontrollable and unobservable eigenvalues of the original system which are necessary and sufficient for 
objectives (\ref{eq:Objective1}) and (\ref{eq:Objective2})  to  be achieved.
First, we develop a useful preliminary result.

\subsection{A preliminary result}

Here we provide an initial characterization (Lemma \ref{lem:prelim0}) of the systems for which the desired objectives can be achieved.
This result alone can be useful in determining whether or not  the desired objectives can be achieved.
First, we need  the following result  on SPR transfer functions which are not necessarily proper.

\begin{lemma}
\label{lem:spr}
A rational SPR transfer function $H$ has the following properties.
\begin{itemize}
\item[(a)] If  infinity is a pole of $H$  then, its order  is  one; also
$H_1$ is symmetric and positive semi-definite where
$
H_1 = \lim_{ s\rightarrow \infty} s^{-1}H(s)
$.
\item[(b)] The finite poles of $H$ have negative real part.
\item[(c)] $H^{-1}$ is SPR.
\end{itemize}
\end{lemma}
The appendix contains a proof.

We have now the following preliminary result.
\begin{lemma}
\label{lem:prelim0}
Given $L$ non-singular,  there exists a matrix $K$, with $I-KD$ non-singular,  such that $L(I-GK)^{-1}G$ is SPR if and only if 
the following conditions hold.
\begin{itemize}
\item[(a)]
If  infinity is a pole of $G^{-1}$  then, its order  is  one; also
$L'H_1$ is symmetric and positive semi-definite where
\begin{equation}
H_1 = \lim_{ s\rightarrow \infty} s^{-1}G(s)^{-1}
\end{equation}
\item[(b)]
The finite poles of $G^{-1}$ have negative real part.
\item[(c)]
There exists  a  matrix $M$   such that
\begin{equation}
\label{eq:L'GinvIneq}
L'G( \jmath \omega)^{-1}+ G( \jmath \omega) ^{-\prime} L+M +M'  >0 \
\end{equation}
for $-\infty \le \omega \le \infty$, 
and $I+L^{-\prime}MD$ is non-singular.
\end{itemize}

The matrix $K$ is given by
\begin{equation}
\label{eq:K=-LM}
K=-L^{-\prime}M
\end{equation}

\end{lemma}

\begin{proof}
By Lemma \ref{lem:spr}, the transfer function
 $G_c = L(I-GK)^{-1}G$   is  SPR if and only if  $G_c^{-1}$ is SPR.
Since 
 $L$ is non-singular, 
\begin{equation}
\label{eq:Kcondition}
G_c^{-1} = G^{-1}(I-GK)L^{-1}   = G^{-1}L^{-1} - KL^{-1}
\end{equation}
Also, $G_c^{-1}$ being SPR is equivalent to
\begin{equation}
\label{eq:L'GCL}
 L'G_c^{-1}L = L'G^{-1} -L'K
 \end{equation}
being SPR.\newline

Now suppose that  $L'G_c^{-1}L$ is SPR.
It follows from \eqref{eq:L'GCL} that $L'G_c^{-1}L$ and $G^{-1}$ have the same poles and the order of infinity as a pole of
$L'G_c^{-1}L$  and  $G^{-1}$  is the same.
Also
$\lim_{s \rightarrow \infty} s^{-1} L'G_c^{-1}L = L'H_1$.
Thus properties (a) and (b) of the lemma hold.
Since $L'G_c^{-1}L $ is  SPR, it has no poles on the imaginary axis and, using \eqref {eq:L'GCL},
\begin{equation}
\label{eq:L'G}
L'G(\jmath \omega)^{-1} + G(\jmath \omega)^{-\prime}L  -L'K - K'L  >0 
\end{equation}
for $-\infty < \omega < \infty$.
If we express $G^{-1}$ as
\begin{equation}
G^{-1}(s) = s H_1 + H_0(s)
\end{equation}
then $H_0$ does not have a pole at infinity.
Since $L'H_1$ is symmetric, 
\begin{align}
&L'G(\jmath \omega)^{-1}  + G(\jmath \omega)^{-\prime} L 
\nonumber\\
&= \jmath \omega L'H_1 - \jmath \omega L'H_1 +L'H_0(j\omega) + H_0(\jmath \omega)'L		\nonumber \\
&=
L'H_0(j\omega) + H_0(\jmath \omega)'L
\label{eq:GinvH0}
\end{align}
 and \eqref{eq:L'G} is equivalent to
 \begin{equation}
 L'H_0(j\omega) + H_0(\jmath \omega)'L    -L'K - K'L  >0  
 \end{equation}
for $-\infty < \omega < \infty$.
 Hence
 \begin{equation}
 L'H_0(j\omega) + H_0(\jmath \omega)'L    -L'K - K'L  \ge 0  
 \end{equation}
 for $-\infty < \omega < \infty$.
 Consider now  any $M$ satisfying $M+M' > -L'K - K'L $  with $I +L^{-\prime}MD$ non-singular.
 This results in 
  \begin{equation}
  \label{eq:ineqLH0}
 L'H_0(\jmath \omega) + H_0(\jmath \omega)'L   +M +M'  > 0  
 \end{equation}
 for $-\infty < \omega < \infty$.
 Combining this with  \eqref{eq:GinvH0} yields the desired inequality \eqref{eq:L'GinvIneq} in (c).\newline

 Now suppose that  (a)-(c) hold.
 Inequality  \eqref{eq:L'GinvIneq}  and \eqref{eq:GinvH0} imply  \eqref{eq:ineqLH0}.
 Now choose $\beta >0$ small enough so that
 \begin{equation}
   \label{eq:ineqLH02}
 L'H_0(\jmath \omega) + H_0(\jmath \omega)'L   +M +M'  -2 \beta I > 0  
 \end{equation}
 for $-\infty < \omega < \infty$.
  It follows from (b) that all the finite poles of  $L'H_0$ have negative real parts.
  This along with \eqref{eq:ineqLH02}  implies that $L'H_0 +M -\beta I$ is SPR; thus there exists $\epsilon_1 >0$ such that
  \begin{equation}
  L'H_0(s) + H_0(s)'L +M +M' - 2\beta I >0   \end{equation}
for $\Re(s) \ge - \epsilon_1$.
 Consider any $\epsilon_2 >0$ for which  $\epsilon_2 H_1 \le \beta I$.
 Letting $\epsilon = \min\{\epsilon_1, \epsilon_2\}$ and recalling that $L'H_1$ is symmetric positive semi-definite, we obtain that
  whenever $\Re(s) \ge - \epsilon$,
 \begin{align*}
& L'G(s)^{-1} +M + G(s)^{-\prime}L +M' \\
&= (s +\bar{s})L'H_1 + L'H_0(s) + H_0(s)'L +M+M' \\ \nonumber
 &> -2 \epsilon_2 H_1 +2 \beta I   \ge 0 \nonumber
 \end{align*}
This implies that $L'G(s)^{-1} +M$ is SPR.
Letting $K=-L^{-\prime}M$, we have $M= -L'K$. Hence $L'G^{-1} -L'K$ is SPR.
Without loss of generality we can consider $M$ sufficiently large so that $I+L^{-\prime}MD$ is non-singular.
(To see this, consider $M= \kappa I$ and note that  $ I+\kappa L^{-\prime}D$ is singular only when $\kappa=-\alpha^{-1}$  where $\alpha$ is a nonzero eigenvalue of $ L^{-\prime}D$.)
Thus $I-KD = I+L^{-\prime}MD$ is non-singular.

\end{proof}

\begin{ex}{\rm
Recalling the transfer function in Example \ref{ex5}, $G^{-1}$ has no infinite pole and a single finite pole at $-1$.
Hence $H_1=0$ and $L$ can be any non-negative number. Also $D=1$ and 
$G(j\omega)^{-1} + G(j\omega)^{-'} = 2-6/(1+\omega^2)$.
With $L=1$,  inequality   \eqref{eq:L'GinvIneq} is satisfied if $M>2$; hence, using \eqref{eq:K=-LM} any $K<-2$ will render an SPR closed system.
Similarily $L=-1$ results in  $K>1$.
}
\end{ex}

\begin{remark}
Note that inequality \eqref{eq:L'GinvIneq}  is equivalent to $M>M_0$ and
\[
LG(\jmath \omega) +  G(\jmath \omega)^{\prime}  L^\prime + G(\jmath\omega)(M_0+M'_0) G(\jmath\omega)'  \ge 0
\]
for $-\infty  <  \omega <  \infty$.

\end{remark}

\subsection{Proof of Theorem \ref{th:main2}}
We first obtain the following result from Lemma \ref{lem:prelim0}.
\begin{lemma}
\label{lem:prelim02}
Given a square rational  transfer function $G$, there  exist matrices $L$ and $K$, with 
  $I-KD$ non-singular,  such that $L(I-GK)^{-1}G$ is SPR if and only if 
the following conditions hold.  
\begin{itemize}
\item[(a)] 
If  infinity is a pole of $G^{-1}$,   then its order  is  one.
\item[(b)] The finite poles of $G^{-1}$ have negative real part.
\end{itemize}
\end{lemma}

\begin{proof}
The necessity of conditions (a)-(b) follows from Lemma \ref{lem:prelim0}.

We now prove sufficiency of conditions (a)-(b).
When conditions (a) and (b) hold,
\begin{equation}
\label{eq:invG=sH1+}
G(s)^{-1} = sH_1+H_0(s)
\end{equation}
where $H_0$ has no pole at infinity and all its poles have  negative real part.
 Let $L$ be any non-singular matrix for which 
$L^\prime H_1
$
is positive semi-definite.
 (The  construction of $L$ is described in Section \ref{subsec:controller}.)
Since $H_0$ has no imaginary poles  and no pole at infinity, there exists $\kappa_0$ such that
\begin{equation}
\label{eq:beta0}
L^\prime H_0(\jmath \omega) +H_0(\jmath \omega)^\prime L \ge - \kappa_0 
\end{equation}

for  $-\infty \le \omega \le \infty$,
Now choose any  $\kappa > \kappa_0/2$ such that $ I+\kappa L^{-\prime}D$ is non-singular and let $M= \kappa I$
%
It now follows from \eqref{eq:GinvH0} and \eqref{eq:beta0} that
\begin{equation}
\label{tildeH0plusBetaSPR}
L^\prime G(\jmath \omega)^{-1}  + G(\jmath \omega)^{-\prime} L + M +M^\prime   >0
\end{equation}
for  $-\infty \le \omega \le \infty$,
and $ I+ L^{-\prime}MD$ is non-singular.
It now follows from Lemma \ref{lem:prelim0} that
with $K=-L^{-\prime}M$,  the transfer function $L(I-GK)^{-1}G$ is SPR and $I-KD$ is non-singular.

\end{proof}

\begin{remark}[SISO systems]
For a SISO (scalar input scalar output)  system, $G(s)$ is a scalar; hence $\lambda$ is a pole of $G^{-1}$ if and only if
$\lambda$ is a {\sf zero} of $G$, that is, $G(\lambda) = 0$. Also infinity is a pole of $G^{-1}$ if and only if $G$ is strictly proper;
in this case,  infinity is a pole  of order $l$ of $G^{-1}$ if and only if $\infty$ is a zero
of order $l$ of $G$, that is,
$l\ge 1$ is the smallest integer for which
\[
\lim_{s \rightarrow \infty} s^l G(s) \neq 0.
\]
This is the same as the relative degree of the transfer function.
So, for SISO systems, conditions (a) and (b) of Lemma \ref{lem:prelim02}   are respectively equivalent to:
{\em
\begin{itemize}
\item[(a)] If $G$ is strictly proper, its relative degree is one.
\item[(b)] The finite zeros of $G$ have negative real part.
\end{itemize}
}
\end{remark}


%




The following result provides a relationship between the poles and eigenvalues of a descriptor system.
A proof can be found in the Appendix.
\begin{lemma}
\label{lem:descriptor2}
For   system (\ref{eq:ss1}),
$\lambda$ is an eigenvalue of $(E,A)$
if and only if it  is a finite pole,
an uncontrollable eigenvalue or an unobservable eigenvalue.
\end{lemma}

\begin{lemma}
\label{lem:prelim1}
Consider a system described by (\ref{eq:ss1}) and  matrices $K$ and $L$  where $I-KD$ is non-singular.
%
Then objectives (\ref{eq:Objective1}) and (\ref{eq:Objective2}) hold if and only if  the following conditions hold.
\begin{itemize}
\item[(a)]
$L(I-GK)^{-1}G$ is SPR
\item[(b)]
The uncontrollable and unobservable  eigenvalues of the system  have negative real part.
\end{itemize}
\end{lemma}
\vspace{1em}

\begin{proof}
We first prove sufficiency of conditions (a)-(b).
Since   $G_c=L(I-GK)^{-1}G$ is SPR,
all its poles have negative real part.
One can readily show that
 static output feedback does not change uncontrollable and unobservable eigenvalues;
hence condition (b) implies that the uncontrollable eigenvalues of $(E, A_c, B)$ and  the unobservable eigenvalues
of $(E, C, A_c)$ have negative real part. 
It now follows from Lemma \ref{lem:descriptor2} that all the eigenvalues of $(E, A_c)$ have negative real part.
Hence  $(E, A_c)$ is stable. To prove necessity of conditions (a)-(b),  we simply note that the stability of $(E, A_c)$ 
implies
condition (b).
\end{proof}
Theorem \ref{th:main2} is now a consequence of Lemmas \ref{lem:prelim02} and
 \ref{lem:prelim1}.

\ignore{
\begin{cor}
\label{lem:basic}
There exist matrices $K$ and $L$, with $I-KD$ non-singular,  such that objectives (\ref{eq:Objective1}) and (\ref{eq:Objective2})  can be achieved  for system $(E, A, B, C, D)$ if and only if all  the following conditions hold.
\begin{itemize}
\item[(a)] The finite poles of $G^{-1}$ have negative real part where $G$ is the system transfer function.
\item[(b)] If $G^{-1}$ has a pole at infinity then, its order  is one.
\item[(c)]
The uncontrollable and unobservable  eigenvalues of the system   have negative real part.
\end{itemize}
\end{cor}
}

\section{SPRification in state space}


\subsection{A new descriptor system }

The key innovation in obtaining Theorem \ref{th:main} is to express $G^{-1}$ as the transfer function of a new descriptor system  described by
$(\mathcal{E}, \mathcal{A}, \mathcal{B}, \mathcal{C}, 0)$ where

\begin{equation}
\label{eq:DescNew2}
\mathcal{E} =\left[\begin{array}{cc}
E  &0  \\
0   &0
\end{array}
\right],
\qquad 
\mathcal{A} =\left[\begin{array}{cc}
A &B  \\
C   &D
\end{array}
\right]
\end{equation}
and
\begin{equation}
\label{eq:CalBcalC}
\mathcal{B} =\left[\begin{array}{c}
0\\
I
\end{array}
\right],
\qquad 
\mathcal{C} =\left[\begin{array}{cc}
0   &		-I
\end{array}
\right]
\end{equation}
The transfer function for this system is
\begin{equation}
\label{eq:mathcalG}
\mathcal{G}(s) = \mathcal{C}(s\mathcal{E}-\mathcal{A})^{-1} \mathcal{B}
\end{equation}
and we have the following  result.
\begin{lemma}
\label{lem:main2}
\begin{itemize}
\item[(a)]
\begin{equation}
\label{eq:GinvCalG}
G^{-1} = \mathcal{G}
\end{equation}
\item[(b)]
The uncontrollable eigenvalues of $(E, A,B)$ and $(\mathcal{E}, \mathcal{A}, \mathcal{B})$ are the same.
\item[(c)]
The unobservable eigenvalues of $(E,C,A)$ and $(\mathcal{E}, \mathcal{C}, \mathcal{A})$ are the same.

\end{itemize}
\end{lemma}

\begin{proof}
Recall that
$
G(s) = C(sE-A)^{-1}B + D$.
With  $\mathcal{E}$ and $\mathcal{A}$ given by (\ref{eq:DescNew2}),
we see that
\begin{align*}
&\left(\mathcal{A}-s \mathcal{E}\right)^{-1} =
\left[
\begin{array}{cc}
A-sE    &B  \\
C       &D
\end{array}
\right]^{-1}= \\
&\left[
\begin{array}{cc}
 *   &  *  \\
  *     &\left[D+C(sE-A)^{-1}B\right]^{-1}
\end{array}
\right]
 =
\left[
\begin{array}{cc}
*   &*  \\
*   &G(s)^{-1}
\end{array}
\right]
\end{align*}
Hence, recalling (\ref{eq:CalBcalC}), we see that
$
G(s)^{-1} = \mathcal{C}(s\mathcal{E} - \mathcal{A})^{-1}\mathcal{B}
={\mathcal G}(s)$.

To prove (b), note that
\begin{eqnarray*}
\mbox{rank}\left[
\begin{array}{cc}
\mathcal{A} - \lambda \mathcal{E}& \mathcal{B}
\end{array}
\right]
&=&
\mbox{rank} \left[
\begin{array}{ccc}
A - \lambda E& B        &0  \\
C              &D       &I
\end{array}
\right]\\
&=&
\mbox{rank} \left[
\begin{array}{ccc}
A - \lambda E& B        &0  \\
0             &0      &I
\end{array}
\right]
\end{eqnarray*}
Thus, 
 $\left[
\begin{array}{cc}
\mathcal{A} - \lambda \mathcal{E}& \mathcal{B}
\end{array}
\right]$
does not have maximum rank if and only if the same is true for $\left[\begin{array}{cc}
A - \lambda E& B
\end{array}
\right]$.
This means that  $\lambda$ is an uncontrollable eigenvalue of $(E, A, B)$ if and only if it is an uncontrollable eigenvalue of $(\mathcal{E}, \mathcal{A}, \mathcal{B})$.
To prove (c), note that
\begin{eqnarray*}
\mbox{rank}\left[
\begin{array}{c}
\mathcal{A} - \lambda \mathcal{E}\\ \mathcal{C}
\end{array}
\right]
&=&
\mbox{rank} \left[
\begin{array}{ccc}
A - \lambda E& B      \\
C              &D     \\
0   & -I
\end{array}
\right]\\
&=&
\mbox{rank} \left[
\begin{array}{ccc}
A - \lambda E   &   0  \\
C                   &0\\
0              &I
\end{array}
\right]
\end{eqnarray*}
Thus, $\left[
\begin{array}{c}
\mathcal{A} - \lambda \mathcal{E}\\ \mathcal{C}
\end{array}
\right]$ 
does not have maximum rank if and only if the same is true for $\left[
\begin{array}{c}
A - \lambda E    \\
C                   
\end{array}
\right]
$.
This means that $\lambda$ is an unobservable eigenvalue of $(E, C, A)$ if and only if it is an 
unobservable eigenvalue of $(\mathcal{E}, \mathcal{C}, \mathcal{A})$. 
\end{proof}

\subsection{Finite poles of $G^{-1}$ and the eigenvalues of $(\mathcal{E}, \mathcal{A})$}
Consideration of Lemmas \ref{lem:main2} and \ref{lem:descriptor2} yields the following result
on the finite poles of $G^{-1}$.

\begin{cor}
\label{cor:cor2}
$\lambda$ is 
an eigenvalue of $(\mathcal{E},\mathcal{A})$
if and only if it is
a finite pole of $G^{-1}$  or
an
uncontrollable  or unobservable eigenvalue for $(E,A, B, C, D)$.


\end{cor}

\begin{remark}[SISO systems]
When $G(s)$ is scalar and $s$ is not an eigenvalue of $(E,A)$,
\begin{align*}
&\det
\left(
\mathcal{A} - s \mathcal{E}
\right) = 
\det\left(\left[\begin{array}{ccc}
A - s E   &   B \\
C                   &D
\end{array}
\right]
\right)
\\
&=
\det(A-sE) \det\left(D+C(sE-A)^{-1}B \right)\\
&=
\det(A-sE)G(s)
\end{align*}
Hence
\[
G(s)^{-1} =\frac{\det(A-sE)}{\det(\mathcal{A} - s \mathcal{E}) }
\]
and we have the following conclusion for SISO systems which are controllable and observable.

\vspace{1em}
{\em $\lambda$ is a finite pole of $G^{-1}$ if and only if $\lambda$ is an eigenvalue of $(\mathcal{E}, \mathcal{A})$ and,
if  $\lambda$ is also an eigenvalue of $(E, A)$, its  algebraic multiplicity as an eigenvalue of $(\mathcal{E}, \mathcal{A})$
is greater than its algebraic multiplicity as an eigenvalue of $(E,A)$.}
\end{remark}


\subsection{Infinite poles of $G^{-1}$ and the zero eigenvalues of $\mathcal{A}^{-1}\mathcal{E}$}
To characterize an infinite pole of $G^{-1}$ we need the following result for a general descriptor system
described by $(E, A)$.

\begin{lemma}
\label{lem:DesPhi}
Suppose $A$ is non-singular and $(E, A)$ has index $l\ge1$. 
Then,
\[
(sE-A)^{-1} = s^{l-1}\Phi_{l-1} +  s^{l-2}\Phi_{l-2} +\cdots +  \Phi_0 + T(s)
\]
where $\Phi_{l-1} \neq 0$ and $T$ is a rational function with  $\lim_{s \rightarrow \infty } T(s) =0$.
Also,
\begin{equation}
\label{eq:EPhi}
E\Phi_{l-1} = 0  \qquad \mbox{and} \qquad \Phi_{l-1}E = 0 
\end{equation}
\end{lemma}
\begin{proof}
Appendix.
\end{proof}



\vspace{1em}
\begin{lemma}
\label{lem:InfintyPole}
Suppose $\mathcal{A}$ is non-singular and rank condition (\ref{eq:fullRankInf}) holds.
Then, the order of infinity  as a pole at  $G^{-1}$ is $l$ where 
the index of $(\mathcal{E}, \mathcal{A})$ is $l+1$.

\end{lemma}
\begin{proof}

%
If 
 $E=0$ then, $\mathcal{E}=0$; hence    $(\mathcal{E}, \mathcal{A})$ is  index one.
 Also,   $A$ must be   non-singular since $(E, A)$  is regular.
Since $\mathcal{A}$ is non-singular, $G(s) = D-CA^{-1}B$ must be non-singular.
Hence $G^{-1}$ has no poles at infinity, that is, the order of infinity as a pole of $G^{-1}$ is zero. 

If $E\neq0$,
let $(X, Y)$ be any full rank decomposition of $E$.
Then  $E=XY^{\prime}$ and 
\begin{equation}
\mathcal{E} = \mathcal{X} \mathcal{Y}^\prime
\end{equation}
where
\begin{equation}
\label{eq:XsYs}
\mathcal{X} = \left[\begin{array}{c}
X\\0
\end{array}\right]
\qquad
\mbox{and} \qquad
\mathcal{Y} = \left[\begin{array}{c}
Y\\0
\end{array}\right]
\,.
\end{equation}
Using the matrix inversion formula $(M+UNV)^{-1}=$
 \[
  M^{-1} -M^{-1}U(N^{-1} +VM^{-1}U)^{-1}VM^{-1}\]
  we see that
\begin{align*}
&(s\mathcal{E} - \mathcal{A})^{-1}  = (- \mathcal{A} +s\mathcal{X}\mathcal{Y}' )^{-1}		\nonumber \\
&= - \mathcal{A}^{-1} - s\mathcal{A}^{-1}\mathcal{X}(I- s \mathcal{Y}' \mathcal{A}^{-1}\mathcal{X})^{-1}\mathcal{Y}'\mathcal{A}^{-1}
\end{align*}
Then, recalling \eqref{eq:GinvCalG} and \eqref{eq:mathcalG},
\begin{equation}
\label{eq:calGandGtilde}
G(s)^{-1} = \mathcal{G}(s) = -\mathcal{C}\mathcal{A}^{-1}\mathcal{B} + s \tilde{\mathcal G}(s)
\end{equation}
where
\begin{equation}
\label{eq:calG}
 \tilde{\mathcal G}(s) = C_1(sE_1 - I)^{-1}B_1
\end{equation}
and
\begin{align*}
E_1 = \mathcal{Y}^\prime\mathcal{A}^{-1} \mathcal{X},\:
B_1 = \mathcal{Y}^\prime\mathcal{A}^{-1} \mathcal{B},\:	
  C_1 =\mathcal{C}\mathcal{A}^{-1} \mathcal{X} 
\end{align*}
Letting
\begin{equation}
\label{eq:invcalA}
\left[\begin{array}{cc}
\tilde{A}	&\tilde{B}	\\
\tilde{C}	&\tilde{D}
\end{array}
\right]
=
\mathcal{A}^{-1}
\end{equation}
where $\tilde{A}$ has the same dimensions as $A$,
and recalling the expressions for  $\mathcal{X}, \mathcal{Y}$ and $ \mathcal{B}, \mathcal{C}$   in  (\ref{eq:XsYs}) and (\ref{eq:CalBcalC}) results in 
\begin{equation}
\label{eq:tilde3}
E_1=Y^\prime \tilde{A}X\,,
B_1 =	Y^\prime \tilde{B}\,,
C_1 = -\tilde{C}X\,, \mathcal{C}\mathcal{A}^{-1}\mathcal{B} = \tilde{D}
\end{equation}
 If the index  of $(\mathcal{E}, \mathcal{A})$ is $l+1$, Lemma \ref{lem:reduce} tells us that the index of  $(E_1, I)$ is $l$.
 If $l=0$, $E_1$ must be nonsingular and it follows from \eqref{eq:calGandGtilde} and \eqref{eq:calG} that $G^{-1}$ has no pole at infinity.
 Considering $l\ge1$,
it follows from \eqref{eq:calGandGtilde}  that  infinity is a pole of order $l$ of $G^{-1}$  if and only if infinity is a pole of order $l\!-\!1$ of $\tilde{\mathcal G}$.

Thus to complete the proof we need to show that the order of infinity as a pole of 
  $\tilde{\mathcal G}$ is     $l\!-\!1$.
 Since $\tilde{\mathcal G}(s) = C_1(sE_1 - I)^{-1}B_1$ and the  index of $(E_1, I)$ is $l$, it follows from Lemma \ref{lem:DesPhi} that 
\begin{eqnarray}
\tilde{\mathcal G}(s) &=&  s^{l-1}C_1\Phi_{l-1}B_1 +  s^{l-2}C_1\Phi_{l-2}B_1 \nonumber \\ & & +\cdots  +  C_1 \Phi_0B_1 + \tilde{T}(s)
\end{eqnarray}
where $\Phi_{l-1} \neq 0$ and $\tilde{T}$ is a rational function with  $\lim_{s \rightarrow \infty } \tilde{T}(s) =0$.
Also,
\begin{equation}
E_1\Phi_{l-1} = 0 \qquad \mbox{and} \qquad \Phi_{l-1} E_1  = 0
\end{equation}
Clearly, the order of  infinity as a pole of  $\tilde{\mathcal G}$ is $l\!-\!1$ if and only if $C_1\Phi_{l-1}B_1  \neq 0$.
Suppose, on the contrary that
$
C_1\Phi_{l-1}B_1  =0
$.

If $Z:= \Phi_{l-1}B_1 \neq 0$ then   
\begin{equation}
\label{eq:EZ}
E_1Z=0\,, \qquad C_1 Z =0\,, \qquad Z\neq 0 \,.
\end{equation}

 Since $X$ is full column rank and $E=XY'$, it follows from \eqref{eq:EZ} that 
 \[
 E \tilde{A}z=0\,, \qquad 
\tilde{C}z= 0\,, \qquad z =XZ\neq0 \,.
\]
Also,  \eqref{eq:invcalA}  implies that
\[
C\tilde{A}+D\tilde{C} = 0 
 \]
Hence $C\tilde{A}z = - D\tilde{C}z =0$.
Since $\tilde{B}z= 0$ we cannot have $\tilde{A} z=0$.
Thus,
\[
E(\tilde{A}z )=0\,, \qquad  C(\tilde{A}z) = 0\,, \qquad \tilde{A}z \neq 0 \,
\]
that is,  the matrix $\left[\begin{array}{c} E\\C\end{array}\right]$ does not have full rank in contradiction of 
rank condition  (\ref{eq:fullRankInf}).

Now suppose that $\Phi_{l-1}B_1 =0 $; then
we have
\[
\Phi_{l-1}E_1 =0\,, \qquad \Phi_{l-1}B_1 =0 \,, \qquad \Phi_{l-1} \neq 0 \,.
\]
 Since $Y$ is full column rank and $E=XY'$ we now  have,  upon  recalling \eqref{eq:tilde3}, that
\[
z' \tilde{A}E  =0\,, \qquad 
z' \tilde{B} =0 
\qquad \mbox{where} \quad 
z= Y\Phi_{l-1}' \neq 0
\,.
\]
It follows from \eqref{eq:invcalA} that
\[
\tilde{A}B +\tilde{B}D = 0\,.
 \]
Hence $z'\tilde{A}B = - z'\tilde{C}D  =0$.
Since $z'\tilde{B}= 0$,  we cannot have $z'\tilde{A} =0$.
Thus,
\[
(z'\tilde{A} )E=0\,, \qquad  (z'\tilde{A})B = 0\,, \qquad z'\tilde{A} \neq 0\,,
\]
that is, the matrix $\left[\begin{array}{cc} E& B\end{array}\right]$ does not have full rank in contradiction of 
rank condition  (\ref{eq:fullRankInf}).
It now follows that $C_1\Phi_{l-1}B_1 \neq 0$; hence the order  of infinity as a pole of $\tilde{\mathcal G}$ is $l-1$.
\end{proof}

\ignore{
\paragraph{ original system.}
If the original system is normal ($E=I$) we will obtain that
$
G(s^{-1})^{-1} = \tilde{C}(s I - E_1)^{-1} \tilde{B} + \tilde{D}
$
where
\begin{equation}
\left[\begin{array}{cc}
E_1 	&\tilde{B}	\\
\tilde{C}	&\tilde{D}
\end{array}
\right]=
\left[\begin{array}{cc}
A	&B	\\
C	&D
\end{array}
\right]^{-1}
\end{equation}
}


Theorem \ref{th:main} is simply a consequence of 
Theorem \ref{th:main2},
 Corollary
  \ref{cor:cor2}  and  Lemma   \ref{lem:InfintyPole}.

\section{Proof of Theorem \ref{th:main3} and controller construction}

\subsection{Proof of Theorem \ref{th:main3}}
\begin{proof}
It follows from   Lemma \ref{lem:prelim02} that there exist matrices $L$ and $K$ 
with $I-KD$ non-singular 
such that $L(I-GK)^{-1}G$ is SPR if and only if 
$G^{-1}$ can be expressed as
\begin{equation}
\label{eq:GinvR2}
G(s)^{-1} = s H_1 + D_2 + R(s)
\end{equation}
where $H_1$  and $D_2$ are  constant matrices and 
either
\begin{itemize}
\item[(i)] $R=0$
or
\item[(ii)] $R(s) =  C_2(sI-A_2)^{-1}B_2$  where $A_2, B_2, C_2$ are constant matrices with  $A_2$  Hurwitz.
\end{itemize}
Using
  Lemma \ref{lem:prelim0},
  in either case,
     $L^\prime H_1$ is symmetric, positive semi-definite and there exists a matrix $M$ such that
\begin{equation}
\label{eq:L'GinvIneq2}
L'G( \jmath \omega)^{-1}+ G( \jmath \omega) ^{-\prime} L+M +M'  >0 
\end{equation}
for $ -\infty \le \omega \le \infty$.
Moreover,
\begin{equation}
\label{eq:K2}
K=-L^{-\prime}M 
\end{equation}
Since $L^\prime H_1$ is symmetric,   
it follows from 
\eqref{eq:GinvR2}  that
\eqref{eq:L'GinvIneq2}  is equivalent to
\begin{equation}
\label{eq:L'GinvIneq3}
L'R(\jmath \omega) + R(\jmath \omega)'L  + N+N^\prime >0 
\end{equation}
for $ -\infty \le \omega \le \infty$, where $N:= L'D_2  + M$.
Also, using \eqref{eq:K2},
\begin{equation}
\label{eq:Kform}
K =D_2 -L^{-\prime}N \,.
\end{equation}

If $R=0$, \eqref{eq:L'GinvIneq3} is equivalent to
\[
N+N' >0 \,.
\]

If $R(s) =  C_2(sI-A_2)^{-1}B_2$ ,
 \eqref{eq:L'GinvIneq3} is equivalent to
\begin{equation}
\label{eq:G2ineq}
G_2(\jmath \omega) + G_2(\jmath \omega)' >0 
\quad \mbox{for} \quad -\infty \le \omega \le \infty
\end{equation}
where
\begin{equation}
G_2(s) = L'C_2(sI-A_2)^{-1}B_2 +N
\end{equation}
Using the KYP Lemma and Lyapunov theory it follows that
satisfaction of \eqref{eq:G2ineq} is equivalent to the existence of 
a symmetric, positive definite matrix $P$ such that
 \begin{equation}
\label{eq:LMI}
\left[
\begin{array}{cc}
PA_2 +A_2'P  	& PB_2 - C_2' L\\
B_2'P  -L'C_2		&   - N  -N'
\end{array}
\right]
< 0
\end{equation}
Letting
\begin{equation}
Q:= -PA_2 - A_2'P
\end{equation}
and using a 
Schur complement result, inequality \eqref{eq:LMI} is equivalent to
\begin{align}
Q&>0\\
N+N' &>  (PB_2 - C_2' L)'Q^{-1}(PB_2 - C_2' L)
\end{align}
The proof is completed by invoking Lemma \ref{lem:prelim1}.
\end{proof}

\subsection{Controller construction}
\label{subsec:controller}
We provide here a method for obtaining the matrices $A_2, B_2, C_2, D_2$ and $H_1$ used in controller construction.
This method uses the state space description $(E,A,B,C,D)$ of the original system.

If $E=0$ then $G(s)^{-1}  =[D-CA^{-1}B]^{-1}$ and $H_1=0, R=0$ and $D_2=[D-CA^{-1}B]^{-1}$.

If $E\neq0$, let
$(X, Y)$ be any full rank decomposition of $E$, that is
   \begin{equation}
   E=XY^{\prime}\end{equation}
    where $X$ and $Y$ are full rank matrices \cite{SaiijaAl2013Descriptor}.
    One method of computing $X$ and $Y$ is to obtain a singular value decomposition
    of $E$, that is $
E= U\Sigma V'
$
where $U$ and $V$ are orthogonal matrices
and $\Sigma$ is diagonal with non-negative elements, $\sigma_1 \ge \sigma_2 \ge  \dots \ge \sigma_n$
called the singular values of $E$.
If $E$ is non-zero, let $\sigma_r$ be the smallest singular value of $E$. 
Then,
\begin{equation}
E=U_1\Sigma_1V_1^\prime
\end{equation}
where $U_1$ and $V_1$ consist of the first $r$ columns of $U$ and $V$ respectively and
$\Sigma_1$ is diagonal with  positive diagonal elements $\sigma_1,  \sigma_2, \dots, \sigma_r$.
    Now let 
  \begin{equation}
  X=U_1\Sigma_1\,, \qquad Y=V_1
  \end{equation}
    
    It follows from \eqref{eq:calGandGtilde}, \eqref{eq:calG} and   \eqref{eq:tilde3} that
\begin{equation}
G(s)^{-1} = \tilde{D} + sC_1(sE_1-I)^{-1}B_1
\end{equation}
where $B_1$, $C_1$  and $E_1$ are given in \eqref{eq:tilde3}.

\begin{itemize}
\item[(i)]
If $E_1=0$ then,
\begin{equation}
G(s)^{-1} = sH_1 +D_2
\end{equation}
where
\begin{equation}
D_2 = \tilde{D}\,,\quad H_1 = -C_1B_1 = \tilde{C}XY'\tilde{B} = \tilde{C}E\tilde{B}
\end{equation}
 
\item[(ii)]
If  $E_1\neq 0$,  we 
let
$(X_1, Y_1)$ be any full rank decomposition of $E_1$; thus
   \begin{equation}
 E_1=X_1Y_1^{\prime}
   \end{equation}
%
Since the index of $(\mathcal{E}, \mathcal{A})$ is at most two,  the matrix $Y_1'X_1$  is invertible.
%
Proceeding, we obtain that
\begin{eqnarray}
\label{eq:des}
G(s)^{-1} &=& sH_1 + H_0(s)\, \\  \label{eq:des2} H_0(s)  &=&D_2 + C_2(sI-A_2)^{-1} B_2
\end{eqnarray}
where
\begin{equation}
H_1= \tilde{C}X(I- X_1A_2Y_1' )Y'\tilde{B}
\end{equation}
and
\begin{equation*}
\begin{array}{ll}
A_2=(Y_1'X_1)^{-1} \, &
B_2 = A_2\,Y_1' Y' \tilde{B}       \\
C_2= -\tilde{C} X\,X_1 A_2^2  \,,  &
D_2 = \tilde{D} +C_2A_2^{-1}B_2    
\end{array}
\end{equation*}
\end{itemize}

\subsubsection{Obtaining $L$}
Let $U\Sigma V'$ be a singular value decomposition of $H_1$.
Then
\begin{equation}
H_1 = U\Sigma V'
\end{equation}
where $U$ and $V$ are orthogonal matrices
and $\Sigma$ is diagonal with non-negative elements. 
With
\begin{equation}
 L=UV'
 \end{equation}
 we have
$
L'H_1 =VU'U\Sigma V' = V\Sigma V' \ge 0
$
and $L'H_1$ is symmetric.

\section{Conclusions}
Conditions are derived to determine
when and how  a system can be made SPR via output feedback. The first is a time domain spectral (eigenvalue) condition; the 
second an equivalent frequency domain condition. Finally,  a control design procedure to make a given system SPR is given. Simple examples are given to illustrate 
our results and to demonstrate their efficacy.

\section{Acknowledgements}
The authors are grateful for  useful discussions with Professor Edward J. Davison of the University of Toronto.

\section{Appendix}

We now give some results that are important for our discussion.
Note that some of these have appeared in preliminary form 
in our previous papers \cite{SaijjaAl2013CommentsDescriptor, CorlessShortenSPR2010, CorlessShorten2011GenEV}

\subsection{Calculation 1}
First note that
$
y=Cx+DKy +Dw
$.
Assuming $I-DK$ is non-singular, 
$
y=(I-DK)^{-1}[Cx+Dw]
$
and
\begin{align*}
u&=K(I-DK)^{-1}[Cx +Dw] +w\\
&=(I-KD)^{-1}KCx +[I+(I-KD)^{-1}KD]w\\
&=(I-KD)^{-1}KCx +(I-KD)^{-1}w
\end{align*}
Hence, the system resulting from (\ref{eq:fbk}) applied to  (\ref{eq:ss1}) is given by
\begin{equation*}
\begin{array}{ll}
E\dot{x} = [A +B(I-KD)^{-1}KC]x +B(I-KD)^{-1}w\\
z= L(I-DK)^{-1}Cx +LD(I-KD)^{-1}w
\end{array}
\end{equation*}


\subsection{Proof of Lemma \ref{lem:reduce}}
 \begin{proof}
First note that $(\mathcal{X}, \mathcal{Y})$ is a full rank decomposition of $\mathcal{E}$
where
\begin{equation}
\label{eq:XsYs0}
\mathcal{X} = \left[\begin{array}{c}
X\\0
\end{array}\right]
\qquad
\mbox{and} \qquad
\mathcal{Y} = \left[\begin{array}{c}
Y\\0
\end{array}\right]
\end{equation}
Hence
\begin{equation}
\label{eq:calAcalE}
\mathcal{A}^{-1}\mathcal{E}= \mathcal{A}^{-1} \mathcal{X}\mathcal{Y}'
\end{equation}
Also
\begin{equation}
\label{eq:tildeA1}
E_1 = \mathcal{Y}'\mathcal{A}^{-1}\mathcal{X}
\end{equation}
  When $s$ is non-zero,
  \begin{eqnarray*}
  \det(sI -\mathcal{A}^{-1}\mathcal{E}) &=& \det(sI - \mathcal{A}^{-1}\mathcal{X}\mathcal{Y}')\\
  &=& s^n\det(I - s^{-1}\mathcal{A}^{-1}\mathcal{X}\mathcal{Y}')\\
 &=& s^n\det(I -s^{-1}\mathcal{Y}'\mathcal{A}^{-1}\mathcal{X})\\
  &=&  
  s^{n-m}\det(sI - \mathcal{Y}'\mathcal{A}^{-1}\mathcal{X})\\
 &=& s^{n-m}\det(sI -E_1)
  \end{eqnarray*}
  where $m$ is the rank of $\mathcal{X}$.
 This tells us that the non-zero eigenvalues of  $\mathcal{A}^{-1}\mathcal{E}$  and $E_1$ are the same.

Recalling \eqref{eq:calAcalE} and \eqref{eq:tildeA1}, we obtain that,  for $k=1,2, \ldots$,
  \[
  (\mathcal{A}^{-1}\mathcal{E} )^k =  \mathcal{A}^{-1}\mathcal{X}(\mathcal{Y}'\mathcal{A}^{-1}\mathcal{X})^{k-1} \mathcal{Y}' 
  = \mathcal{A}^{-1}\mathcal{X}E_1^{k-1} \mathcal{Y}'
  \]
  Since $\mathcal{A}^{-1}\mathcal{X}$ is full column rank  and $\mathcal{Y}'$ is full row rank,
 the rank of $  (\mathcal{A}^{-1}\mathcal{E} )^k $ equals the rank of $E_1^{k-1}$.
 Hence the index of zero as  an eigenvalue of $E_1$ equals $l-1$ where
  the index of zero as eigenvalue of $\mathcal{A}^{-1}\mathcal{E}$ is $l$.  
\end{proof}

\subsection{Proof of Lemma \ref{lem:spr}}
\begin{proof}
(a) Since $H$ is a rational transfer function, we can express it as
\begin{equation}
H(s) = s^lH_l + \cdots + s H_1 + H_0(s)
\end{equation}
where $H_1, \ldots, H_l$ are constant  matrices and the rational function $H_0$ does not have a pole at infinity.
We need to show that $H_2, \ldots H_l =0$.
Since $H$ is SPR, we have
\begin{equation}
\label{eq:GSPR}
H(s) + H(s)^\prime > 0 \qquad \mbox{when} \quad \Re(s) \ge  0 
\end{equation}
and $s$ is finite and not a pole of $H$.
Consider any non-zero vector $u$. Then
$h(s):=u^\prime H(s)u$
can be expressed as 
\begin{equation}
\label{eq:g}
h(s) =h_ls^l + \cdots + h_1s +h_0(s)
\end{equation}
where
\begin{equation}
h_k = u^\prime H_ku \quad \mbox{for} \quad  k=1, \cdots, l \quad \mbox{and}
\quad h_0=u^\prime H_0u\,.
\end{equation}
Also $h_0$ does not have a pole at infinity and 
it follows from (\ref{eq:GSPR}) that 
\begin{equation}
\label{eq:gSPR}
h(s) + h(s)' >0 \qquad \mbox{when} \quad \Re(s) \ge 0 
\end{equation}
and $s$ is finite and not a pole of $H$.

Consider any integer $m \ge 1$ for which $g_k =0$ for $k> m$  and suppose
$s=re^{\jmath \theta}$  where $-\pi/2 \le \theta \le \pi/2$ and $r>0$.
 Then, $r^{-m}  =e^{\jmath m \theta}s^{-m}$
and, recalling (\ref{eq:g}),
\[
\lim_{r \rightarrow \infty} r^{-m}h(re^{\jmath \theta})  
= e^{\jmath m \theta} \lim_{s \rightarrow \infty} s^{-m}h(s)  
= e^{\jmath m\theta} h_m\,.
\]
Since $\Re(r e^{\jmath \theta}) \ge 0$  and $r>0$, it follows from (\ref{eq:gSPR}) that
\[
r^{-m}h(re^{\jmath \theta})   +r^{-m}h(re^{\jmath \theta})^\prime >0\,;  
\]
when $r$ is sufficiently large;
hence, considering limits as $r \rightarrow \infty$,
\[
e^{\jmath m\theta} h_m+ e^{-\jmath m\theta}h_m'  \ge 0
\]
Let $ \alpha = \arg(h_m)$ with $0 \le \alpha < 2\pi$.
 Then $h_m = |h_m|e^{\jmath \alpha}$ and
\begin{align*}
&e^{\jmath m \theta} h_m+ e^{-\jmath m \theta}h_m'= e^{\jmath(\alpha + m\theta)} |h_m| + e^{-\jmath(\alpha + m \theta)} |h_m| \\
& = 2 \cos(\alpha +m \theta)|h_m| \; 
\end{align*}
hence,
\begin{equation}
\label{eq:cos}
\cos(\alpha +m \theta)|h_m| \ge 0 
\end{equation}
If $m >1$, we can choose $-\pi/2\le \theta \le \pi/2$
so that
\begin{equation}
\label{eq:theta}
\pi/2 < \alpha + m \theta < 3\pi/2
\end{equation}
which results in $\cos(\alpha +m \theta) <0$.
It now follows from (\ref{eq:cos}) that we must have $h_m= 0$.
By induction, we obtain that
$
h_k = 0 \quad \mbox{for} \quad k >1
$.

When $m=1$ and $\alpha >0$, one can still choose $\theta$ to satisfy (\ref{eq:theta}). 
Hence $\alpha =0$ for $h_1 \neq 0$;
in this case 
$h_1$ is positive real.
Thus we must have
\[
h_1^\prime = h_1 \ge 0
\]

Since the above holds for any non-zero  complex vector $u$ we obtain the desired result that
\[
H_k = 0 \quad \mbox{for} \quad k> 1
\]
and 
\[
H_1^\prime = H_1\ge 0 \,.
\]

The demonstration that all finite poles of $H$ have negative real part proceeds  in a similar fashion;  see \cite{Guillemin1949}.
For any $\lambda$ with $\Re(\lambda) \ge 0$ the proof 
 proceeds by letting
 \[
 h(s) =h_l(s-\lambda)^{-l}+ \cdots + h_1(s-\lambda)^{-1}+h_0(s)
 \]
 where $h_0$ does not have a pole at $\lambda$ and 
 considering the behavior of \mbox{$r^mh(\lambda + re^{\jmath \theta})$}    as $r \rightarrow 0$.
 
 To show that $H^{-1}$ is SPR let $\epsilon_1 >0$ be such that 
 $H(s) +H(s)'>0$ when $s$ is not a pole of $H$ and $\Re(s) \ge -\epsilon_1$.
 Now choose $\epsilon_2 >0$ such that $\Re(s) <-\epsilon_2$ whenever $s$ is a pole of $H$.
 Letting $\epsilon=\min\{\epsilon_1, \epsilon_2\}$ we see that
 \[
 H(s) + H(s)' >0 \qquad \mbox{for} \quad \Re(s) \ge -\epsilon
 \]
 Thus whenever $\Re(s)\ge -\epsilon$, $H(s)$ is nonsingular and pre- and post-multiplying the above inequality by
 $H(s)^{-1}$ and $H(s)^{-'}$ yields
  \[
 H(s)^{-1} + H(s)^{-' }>0 \qquad \mbox{for} \quad \Re(s) \ge -\epsilon
 \]
 Hence $H^{-1}$ is SPR.
 \end{proof}

\subsection{Proof of Lemma \ref{lem:descriptor2}}

\begin{proof}
Suppose that $\lambda$ is an eigenvalue of $(E, A)$ but not a pole of the transfer function
\[
G(s) = C(sE-A)^{-1}B+D
\]
and define
\begin{equation}
\label{eq:X(s)}
X(s) =(s E - A)^{-1}B
\,.
\end{equation}

First, we show that if 
$\lambda$ is  not a pole of $X$ then it is an uncontrollable eigenvalue of $(E, A, B)$.
When $\lambda$ is a not a pole of $X$,
\begin{equation}
\lim_{s \rightarrow \lambda}  X(s) = X_0
\end{equation}
for some limit $X_0$.
From (\ref{eq:X(s)}) we have
$
(sE-A)X(s) =B
$
and taking the limit of this expression as $s \rightarrow \lambda$ we see that
$
(\lambda E -A)X_0 = B
$.
This implies that
the matrices
$
 [ \lambda E -A \quad  B] $
 and 
$
\lambda E -A
$
have the same rank.
Since $\lambda$ is an eigenvalue
of $( E, A)$, the matrix 
$ \lambda E -A$ does not have full row rank.
Hence $ [ \lambda E- A \quad B]$ does not have max rank and 
$\lambda$  must be  an uncontrollable eigenvalue of
$(E, A, B)$.

Now we show that if   $\lambda$ is a pole of $X$ it must be  an unobservable eigenvalue of $(E, C, A)$.
When $\lambda$ is  not a pole of $X$,
\begin{equation}
\lim_{s \rightarrow \lambda} (s-\lambda)^p X(s) = X_p
\end{equation}
for some $p>0$ and $X_p \neq 0$.
From (\ref{eq:X(s)}) we have
\[
(sE-A)(s-\lambda)^pX(s) =(s-\lambda)^pB
\]
Taking the limit of this expression as $s \rightarrow \lambda$ results in
$
(\lambda E -A)X_p=0
$.
Since,
$
G(s) = CX(s) +D \,,
$
we also have
\[
(s-\lambda)^p G(s) = C (s-\lambda)^p X(s) + (s-\lambda)^p D
\]
Recalling that $\lambda$ is not a pole of $G$ and
taking the limit of the above  expression as $s \rightarrow \lambda$ yields
$
0 = CX_p
$.
Thus,
\[
\left[\begin{array}{c}
\lambda E - A\\
C
\end{array}
\right]X_p= 0
\,.
\]
Since $X_p \neq 0$ this implies that the matrix
\[
\left[\begin{array}{c}
\lambda E - A\\
C
\end{array}
\right]
\]
does not have max rank. Hence $\lambda$ is an unobservable eigenvalue of $(E, C, A)$.
\end{proof}

\subsection{Proof of Lemma \ref{lem:DesPhi}}
\begin{proof}
We first observe that  $\Phi$ defined by $\Phi(s) :=(sE-A)^{-1}$ is a rational function.
Since the index of $(E,A)$ is at least one, $E$ is non-singular and we claim that
$s\Phi(s)$ has a pole at infinity.
This follows from
\[
s(sE-A)^{-1} = (E-s^{-1}A)^{-1}
\]
Hence,
$\Phi(s)$
 can be expressed as
\begin{equation}
\label{eq:Phi0}
(sE-A)^{-1} = s^{l-1}\Phi_{l-1} +  s^{l-2}\Phi_{l-2} +\cdots +  \Phi_0 + T(s)
\end{equation}
for some integer $l\ge 1$ 
where 
$\Phi_{l-1} \neq 0$ 
and $T$ is a rational function with  $\lim_{s \rightarrow \infty } T(s) =0$.
Multiplying both sides of the above equation  on the  left by $sE-A$ yields
{
\begin{align*}
I&=s^lE\Phi_{l-1} +s^{l-1}(E\Phi_{l-2}-A\Phi_{l-1}) +  \cdots \\  
&+ s(E\Phi_0 - A \Phi_1) +sET(s) - A\Phi_0 -AT(s)
\end{align*}
}
for all complex numbers $s$.
Hence
\begin{eqnarray}
& &E\Phi_{l-1}  = 0		\label{eq:Phi1}						\\
& &E\Phi_{k}-A\Phi_{k+1} = 0,	\;  \;\; k=0, .., l-2			\label{eq:Phi2}\\
& &(sE-A)T(s)  -A\Phi_0 = I		\label{eq:Phi3}
\end{eqnarray}

It follows from \eqref{eq:Phi2} that
\begin{equation}
\label{eq:Phi4}
\Phi_{l-1} = (A^{-1}E)^{k}\Phi_{l-1-k}, \;\;  k=0, \ldots, l-1 
\end{equation}
Since $\Phi_{l-1} \neq 0$, it follows that $\Phi_{l-1-k}\neq 0$  for $k=0, \ldots, l-1$.
 Using \eqref{eq:Phi1} and \eqref{eq:Phi4},  we  deduce that
\[
(A^{-1}E)^{k+1}\Phi_{l-1-k}  = 0\,, \qquad (A^{-1}E)^{k}\Phi_{l-1-k} \neq 0\,,  
\]
for $k= 0, \ldots l-1$.
Hence, the rank of $(A^{-1}E)^{k+1}$ is less than the rank of $(A^{-1}E)^{k}$ for $ k= 0, \ldots,  l-1$.
It follows from \eqref{eq:Phi3} that
\begin{eqnarray}
E\Phi_{-1} -A\Phi_0 = I	\nonumber
\end{eqnarray}
where $\Phi_{-1} = \lim_{s \rightarrow \infty }s T(s)$.
Hence
$\Phi_0 =  A^{-1}E\Phi_0 - A^{-1}$ and
\[
0 = (A^{-1}E)^l\Phi_0 =  (A^{-1}E)^{l+1}\Phi_{-1} - (A^{-1}E)^lA^{-1}	
\]
Hence
$(A^{-1}E)^l  = (A^{-1}E)^{l+1}\Phi_{-1}A$.
This means that  $(A^{-1}E)^{l}$   and $(A^{-1}E)^{l+1}$ have the same rank.
Thus $l$ is the index of $(E,A)$. The second equality in \eqref{eq:EPhi} can be obtained 
multiplying both sides of \eqref{eq:Phi0}    on the  right by $sI-A$.
\end{proof}

\end{document}